\documentclass[12pt]{article}
\usepackage{mathrsfs}
\usepackage{amsthm}
\usepackage{amssymb}
\usepackage{latexsym}
\usepackage{amsmath,amsfonts}
\usepackage{mathrsfs}
\usepackage{cases}
\usepackage{latexsym,bm}
\usepackage{indentfirst}
\usepackage{xcolor}
\usepackage[normalem]{ulem}
\usepackage{ifpdf}
\usepackage{graphicx}
\usepackage{epstopdf}
\usepackage{epsfig}
\usepackage{psfrag}
\usepackage{enumitem}
\usepackage{epstopdf}
\usepackage{verbatim}
\usepackage{color}

\usepackage[
pdfauthor={Lu},
pdftitle={Bricks that every removable edge is solitary},
pdfstartview=XYZ,
bookmarks=true,
colorlinks=true,
linkcolor=blue,
urlcolor=blue,
citecolor=blue,
bookmarks=true,
linktocpage=true,
hyperindex=true
]{hyperref}

\title{ Bricks in which every removable edge is solitary
\footnote{ Supported by the National Natural Science Foundation of China (Grant Nos. 12371355, 12271235 and  12601685), Natural Science Foundation of Fujian Province (No. 2026J002034) and Institute of Meteorological Big Data-Digital Fujian.
 
 \footnotesize{\(^{\dagger}\)Corresponding author. }
 }
}

\author{
	\parbox{0.9\textwidth}{%
		\centering
		Jinxin Xue\(^{1}\), Jun Ge\(^{1}\), Fuliang Lu\(^{2,\dagger}\), Yaxian Zhang\(^{1}\) \\[1ex]
		\small
		\(^{1}\)School of Mathematical Sciences, Sichuan Normal University, Chengdu, Sichuan 610068, P.R. China \\
		\small
		\(^{2}\)School of Mathematics and Statistics \& Fujian Key Laboratory of Granular Computing and Applications, Minnan Normal University, Zhangzhou, Fujian 363000, P.R. China
	}%
}

\date{}

\newtheorem{lem}{Lemma}[section]
\newtheorem{thm}[lem]{Theorem}
\newtheorem{cor}[lem]{Corollary}

\newtheorem{pro}[lem]{Proposition}

\newtheorem{claim}{Claim}

\newtheorem{prob}[lem]{Problem}

\definecolor{revblue}{RGB}{0,82,155}

\def\typeone{type~\uppercase\expandafter{\romannumeral 1}}
\def\typetwo{type~\uppercase\expandafter{\romannumeral 2}}

\begin{document}
\newcommand{\udots}{\mathinner{\mskip1mu\raise1pt\vbox{\kern7pt\hbox{.}}
\mskip2mu\raise4pt\hbox{.}\mskip2mu\raise7pt\hbox{.}\mskip1mu}}
\maketitle
\begin{abstract}

A brick is a 3-connected graph $G$ such that $G-u-v$ has a perfect matching for any two distinct vertices $u,v\in V(G)$.
An edge $e$ in a matching covered graph $G$ is removable if $G-e$ is matching covered.
We say that a removable edge $e$ in a brick $G$ is $b$-invariant if $b(G-e)=b(G)=1$, where $b(H)$ denotes the number of bricks in the tight cut decomposition of a matching covered graph $H$.
An edge of a graph is solitary if it lies in precisely one perfect matching.

Lucchesi and Murty proposed the problem of characterizing bricks, distinct from
$K_4$, $\overline{C_6}$ and the Petersen graph, in which every $b$-invariant edge is solitary. Note that every $b$-invariant edge is removable. 
In this paper, we strengthen the condition by requiring that every removable edge is solitary. We show that every simple nonsolid brick satisfying this strengthened condition can be obtained by repeatedly splicing odd wheels (up to multiple edges). Moreover, we construct an infinite family of such bricks that cannot be obtained by splicing a brick and a graph whose underlying simple graph is \(K_4\).
%
\end{abstract}

{\bf Keywords:} \  Brick; Perfect matching; Removable edge; Solitary edge 
	
\section{Introduction}

Graphs considered in this paper are finite and loopless (multiple edges are allowed, where an edge $e$ of a graph $G$ is a
{\em multiple edge} if there are at least two edges of $G$ with the same ends as $e$). For notation and terminology not defined here, we follow \cite{BM08,LP86}. 
Let $G$ be a graph with the vertex set $V(G)$ and the edge set $E(G)$. 
For $\emptyset\neq X\subset V(G)$, let $\overline{X}=V(G)\setminus X$. The \emph{edge cut} of $G$ determined by $X$ is the set of all edges with one end in $X$ and the other in $\overline{X}$, denoted by $\partial_G(X)$. If $X=\{u\}$, we write $\partial_G(u)$ for $\partial_G(\{u\})$. We shall omit the subscript $G$ when the graph is understood. An edge cut $\partial(X)$ is called \emph{trivial} if $|X|=1$ or $|\overline{X}|=1$.

Let $G$ be a graph with a perfect matching. A connected graph is called \textit{matching covered} if it has at least one edge and each edge is contained in some perfect matching.
We denote by $G/(X\rightarrow x)$ the graph obtained from $G$
by contracting $X$ to a single vertex $x$, for brevity, by $G/X$.
The graphs
$G/X$ and $G/\overline{X}$ are the two \emph{$\partial(X)$-contractions} of $G$.
An edge cut  $\partial(X)$ of a matching covered graph $G$ is \emph{separating}
if both $\partial(X)$-contractions of $G$ are matching covered, and is \textit{tight} if  $|\partial(X)\cap M|=1$ for every perfect matching $M$ of $G$.
A matching covered graph $G$ that contains no nontrivial tight
cuts is called a \textit{brick} if $G$ is nonbipartite and a \textit{brace} if it is bipartite.
A brick $G$ is \textit{solid} if every separating cut of $G$ is also tight.
Edmonds, Lov\'asz and Pulleyblank \cite{ELP82} showed that a graph $G$ is a brick if and only if $G$ is 3-connected and bicritical (i.e. $G-u-v$ has a perfect matching for any two distinct vertices $u$, $v\in V(G)$). 
Thus, every brick has at least four vertices and its minimum degree is at least three.


Let $G$ be a matching covered graph. 
Lov\'asz~\cite{Lovasz87} proved
that $G$ can be decomposed
into a unique list of bricks and braces (up to multiple edges) by a
procedure called the tight cut decomposition. Let $b(G)$ denote the number of bricks in this decomposition. 
An edge $e\in E(G)$ is called \textit{removable} if $G-e$ is matching covered. 
A \emph{removable doubleton} of $G$ means a pair of edges $\{e,f\}$ such that $G-e-f$ is matching covered, but neither $G-e$ nor $G-f$ is. 
Lov\'asz~\cite{L1983} proved that every brick, distinct from $K_4$ (a complete simple graph with four vertices) and $\overline{C_6}$ (the triangular prism) has a removable edge. 
Improving Lov\'asz's result, Carvalho et al. \cite{CLM99} showed that each brick $G$ other than $K_4$ and $\overline{C_6}$ has at least $\Delta-2$ removable edges, where $\Delta$ is the maximum degree of $G$. 
Carvalho et al. \cite{CLM12} proved that every solid brick $G$ distinct from $K_4$ has at least $|V(G)|/2$ removable edges.
Kothari et al. \cite{KCLL20} showed that in an essentially 4-edge-connected cubic brick, each edge is either
removable or participates in a removable doubleton.
Wu et al. \cite{LF23} proved that every claw-free brick $G$ with at least 8 vertices has at least
$5|V(G)|/8$ removable edges.

A removable edge of a matching covered graph $G$ is \textit{$b$-invariant} if $b(G-e)=b(G)$. In particular, if $e$ is a $b$-invariant edge of a brick $G$, then $b(G-e)=1$.
A \emph{near-brick} $G$ is a matching covered graph with $b(G)=1$. Obviously, a brick is a near-brick.
An edge of a graph is \emph{solitary} (or \textit{forcing}) if it lies in precisely one perfect matching. 
Carvalho et al. \cite{Carvalho2004} used $b$-invariant edges to characterize extremal matching covered graphs and proposed solitary edges.
 Recently, Lucchesi and Murty proposed the following problem, see Unsolved Problems 1 in \cite{Lucchesi2024}.

\begin{prob}{\rm\cite{Lucchesi2024}}\label{prob}
	Characterize bricks, distinct from $K_4$, $\overline{C_6}$ and the Petersen graph, in which every $b$-invariant edge is solitary.
\end{prob}

For Problem \ref{prob}, Zhang et al. \cite{Zhang24} found that there are only seven cubic bricks, distinct from $K_4$, $\overline{C_6}$ and the Petersen graph, in which every $b$-invariant edge is solitary. Subsequently, Zhang and Wang \cite{ZhangW25-1} characterized several claw-free bricks, distinct from $K_4$ and $\overline{C_6}$, satisfying the same property. Zhang and Lu \cite{ZhangLuNearBipartite} gave a complete characterization of the near-bipartite bricks in Problem \ref{prob}.
Zhou et al. \cite{F2025} further gave a complete characterization of claw-free solid bricks.

For an integer $k\geq3$, the \emph{wheel} $W_k$ is the graph obtained from a cycle $C$ of length $k$ by adding a new vertex $h$ and joining it to all vertices of $C$.
The cycle $C$ is called the \emph{rim} of $W_k$, the vertex $h$ is called its \emph{hub}, and the edges incident with $h$ are called the \emph{spokes} of $W_k$. A wheel $W_k$ is odd if $k$ is odd. The graph $K_4$ is the odd wheel $W_3$. Any vertex of $K_4$ may be chosen as its hub, and the three edges incident with the chosen hub are its spokes. Every edge of $K_4$ lies in
a removable doubleton. For an odd wheel other than $K_4$, it can be checked that no edge on the rim is removable, and  every spoke is removable  (for example, see Exercise 2.2.4 in \cite{Lucchesi2024}).
Zhang and Wang \cite{ZhangW25} proved that every simple solid brick distinct from $K_4$ in which every $b$-invariant edge is solitary is an odd wheel $W_n$ with $n\geq5$.
Since every $b$-invariant edge is removable, every brick in which each removable edge is solitary also satisfies the condition in Problem \ref{prob}. We consider this stronger condition for simple nonsolid bricks.	

Odd wheels also occur in earlier decomposition results. Carvalho et al. \cite{Carvalho2004} showed that every extremal brick can be obtained by splicing graphs whose underlying simple graphs are odd wheels. He et al. \cite{HLX25} obtained a related decomposition for wheel-like bricks. The following theorem gives a splicing decomposition for simple nonsolid bricks in which every removable edge is solitary.

\begin{thm}\label{main}
	Every simple nonsolid brick in which every removable edge is solitary is a splicing of an odd wheel {\rm(}up to multiple edges{\rm)} and a brick in which every removable edge is solitary.
\end{thm}

  In Section \ref{family-graphs}, we construct an infinite family of simple nonsolid bricks $\{G_i \mid i\geq 0\}$ in which every removable edge is solitary. By Theorem \ref{main1}, no $G_i$ can be obtained by splicing a brick and a graph whose underlying simple graph is $K_4$. Thus the odd wheel in Theorem \ref{main} cannot always be chosen to have underlying simple graph $K_4$. In Section \ref{sec:R=S}, we study when the set of removable edges of a brick obtained by splicing coincides with its set of solitary edges.

\section{ Preliminaries }\label{2}

 We begin with some notation. 
 For a subset $X\subseteq V(G)$, let $G[X]$ denote the subgraph of $G$ induced by $X$. For two disjoint nonempty subsets $X,Y\subseteq V(G)$, let $E_G[X,Y]$ denote the set of edges with one end in $X$ and the other in $Y$. In particular, if $X=\{x\}$ and $Y=\{y\}$, we simply write $E_G[x,y]$ for $E_G[X,Y]$. We shall omit the subscript $G$ when the graph is understood. 
 
For a vertex $u\in V(G)$, let $N_G(u)$ denote the set of neighbors of $u$ in $G$. The \emph{degree} of $u$ in $G$, denoted by $d_G(u)$, is the number of edges incident with $u$. When the graph is understood, we write $N(u)$ and $d(u)$ instead. For \(X\subseteq V(G)\), let $o(G-X)$ denote the number of odd components of $G-X$.
Tutte proved the following fundamental theorem in 1947.


 \begin{thm}[Tutte, see \cite{Tutte47}]\label{thm:Tutte}
 	A graph $G$ has a perfect matching if and only if $o(G-X)\le|X|$, for every $X\subseteq V(G)$.
 \end{thm}

Using Tutte's Theorem, one can deduce the following conclusion for matching covered graphs:

 \begin{pro}[Theorem 4.2 in \cite{Lucchesi2024}]\label{pro:separating}
 	An edge cut $C$ of a matching covered graph $G$ is separating if and only if each edge of $G$ lies in a perfect matching that contains precisely one edge in $C$.
 \end{pro}

 \subsection{ The Splicing of Two Graphs and Robust Cuts }

 Let $G$ and $H$ be two vertex-disjoint graphs and let $u\in V(G)$ and $v\in V(H)$ such that $d_G(u)=d_H(v)$.
 Moreover, let $\theta$ be a given bijection between $\partial_G(u)$ and $\partial_H(v)$.
 We denote by $(G(u)\odot H(v))_\theta$ the graph obtained from the union of $G-u$ and $H-v$ by joining, for each edge $e$ in $\partial_H(v)$, the end of $e$ in $H-v$ to the end of $\theta(e)$ in $G-u$;
 and refer to $(G(u)\odot H(v))_\theta$ as a \emph{splicing} of $G$ (at $u$) and $H$ (at $v$), with respect to the bijection $\theta$. 
 In general, a splicing of two graphs $G$ and $H$ depends on the choice of $u$, $v$ and $\theta$.
 With a slight abuse of notation, we will use the same label of the edge (or vertex) in $(G(u)\odot H(v))_{\theta}$ as in $G$ or $H$, and vice versa.
 The following proposition follows directly from the definition of matching covered graphs.

 \begin{pro}[Theorem 2.13 in \cite{Lucchesi2024}]\label{pro:MC_IS_MC}
 	The splicing of two matching covered graphs is also matching covered.
 \end{pro}
A set of vertices $S\subseteq V(G)$ \emph{covers} a set of edges $F\subseteq E(G)$ if every edge of $F$ is incident with some vertex in $S$.
 \begin{thm}{\rm\cite{Carvalho2004}}\label{thm:2-cut}
 	Let $G$ be a matching covered graph, and let $\partial(X)$ be a nontrivial separating
 	cut of $G$ such that both $G/X$ and $G/\overline{X}$ are bricks. Then, $G$ is a brick if and only if no pair
 	of vertices of $G$, one in $X$, the other in $\overline{X}$, covers the set of edges of $\partial(X)$.
 \end{thm}

 For convenience, we extend the definition of removable to edges not in $G$: if $e\notin E(G)$, we simply regard $e$ as a removable edge of $G$. This convention will facilitate the statement of the following lemmas concerning contractions.

 \begin{lem}[Corollary 8.9 in \cite{Lucchesi2024}]\label{lem:re_also_re}
 	Let $G$ be a matching covered graph, and let $C$ be a separating cut of $G$.
 	If an edge $e$ is removable in both $C$-contractions of $G$, then $e$ is removable in $G$.
 \end{lem}

 \begin{lem}{\rm\cite{CLM02II}}\label{thm:removable doubleton}
 	Let $C:=\partial(X)$ be a separating but not tight cut of a matching covered graph $G$ and let $H:=G/(\overline{X}\rightarrow \overline{x})$.
 	Assume that $H$ is a brick, and $R$ is a removable doubleton of $H$.
 	If $R\cap C=\emptyset$ or if the edge of $R\cap C$ is removable in $G/X$ then $R\setminus C$ contains an edge that is removable in $G$.
 \end{lem}

 A separating cut $C$ of a brick $G$ is a {\em{robust cut}} if both $C$-contractions of $G$ are near-bricks.

\begin{lem}{\rm\cite{Carvalho2004,HLX25}}\label{lem:Robust-solid}
	Every nonsolid brick $G$ has a robust cut $\partial(X)$ such that there exists a subset $X'$ of $X$ and a subset $X''$ of $\overline{X}$ such that $G/\overline{X'}$ is a solid brick, $G/\overline{X''}$ is a brick and the graph $H$, obtained from $G$ by contracting $X'$ and $X''$ to single vertices $x'$ and $x''$, respectively, is bipartite and matching covered, where $x'$ and $x''$ lie in different color classes of $H$.
\end{lem}

\begin{lem}{\rm\cite{LuX24}}\label{lem:2-bi-removable}
	Let $\partial(X_i)$ be an edge cut of a brick $G$ such that $G/(\overline{X_i}\rightarrow \overline{x_i})$ is a brick, for $i\in\{1,2\}$, and $(G/(X_1\rightarrow x_1))/(X_2\rightarrow x_2)$ is a matching covered bipartite graph $H$. Then every edge of $\partial_H(x_1)$ is removable in $H$.
\end{lem}

 \subsection{Solitary Edges in a Matching Covered Graph}

\begin{lem}{\rm\cite{DL2026}}\label{lem:H-not-solitary}
	Assume that $\partial(X_i)$ is an edge cut of a simple brick $G$ such that $G/(\overline{X_i}\rightarrow \overline{x_i})$ is a brick, for $i\in\{1,2\}$, and $(G/(X_1\rightarrow x_1))/(X_2\rightarrow x_2)$ is a matching covered bipartite graph $H$. If $|V(H)|\ge4$, then every edge of $E(H)\setminus E_H[x_1,x_2]$ is not solitary in $H$.
\end{lem}

By the definition of solitary edges, we have the following result.
\begin{pro}\label{pro:solitary}
	 Let $C$ be a separating cut of a matching covered graph $G$, and let $G_1$ and $G_2$ be the two $C$-contractions of $G$. For $e \in E(G)$, if $e$ is solitary in $G$, then $e$ is solitary in $G_i$ when $e \in E(G_i)$ for $i=1,2$. 
\end{pro}

\begin{pro}\label{pro:Me-nonre}
Let $e$ be a solitary edge in a matching covered graph $G$, and $M_e$ be the unique perfect matching of $G$ that contains $e$. Then every edge of $M_e\setminus \{e\}$ is nonremovable in $G$.	
\end{pro}	
\begin{proof}
   Suppose that there exists an edge $e'$ of $M_e\setminus \{e\}$ that is removable in $G$. 
   Since $e\in E(G-e')$ and $G-e'$ is matching covered, there exists a perfect matching $M'$ of $G-e'$ containing $e$.
    Note that $M'$ is also a perfect matching in $G$.
    So $M'$ and $M_e$ are two distinct   perfect matchings containing $e$, contradicting the hypothesis that $e$ is solitary in $G$.
\end{proof}

In view of the inheritance of solitary edges between a matching covered graph and its $C$-contractions established in Proposition 5 of \cite{DL2026}, we derive the following corollary concerning removable edges.
\begin{cor}\label{lem:MC}
	Let $G$ be a matching covered graph and let $C$ be a separating cut of $G$.
	Assume that the two $C$-contractions of $G$ are $G_1$ and $G_2$. If every removable edge of $G$ is solitary, then the following statements hold.
	
	(1) For every perfect matching $M_C$ of $G$ such that $|M_C\cap C|\ge 3$, every edge of $M_C$ is nonremovable in $G$.
	
	(2) Let $e\in E(G_1)$, and let $M_e$ be a perfect matching of $G_1$ containing $e$. If $e$ is removable in $G$, then the unique edge in $M_e\cap C$ is solitary in $G_2$.
\end{cor}

\begin{lem}\label{lem:G2re-is-so}
    Let $G$ be a matching covered graph and let $C$ be a separating but not tight cut of $G$. 
    Assume that the two $C$-contractions of $G$ are $G_1$ and $G_2$. 
	If every removable edge of $G$ is solitary and $G_1$ is an odd wheel possibly with multiple spokes, then every removable edge of $G_2$ is solitary. 
\end{lem}
\begin{proof}
	Suppose, to the contrary, that $e$ is a removable but not solitary edge of $G_2$. 
	By Proposition \ref{pro:solitary}, 
	$e$ is not solitary in $G$. 
	If $e$ is removable in $G_1$, then $e$ is removable in $G$ by Lemma \ref{lem:re_also_re}, and hence solitary in $G$, a contradiction. So $e$ is nonremovable in $G_1$ and thus $e\in C$. Let $W_k$ be the underlying odd wheel of $G_1$.
	We claim that $G_1$ has a perfect matching $M$ containing $e$ and an edge $e'\notin C$ that is removable in $G$. Choose a perfect matching $M$ of $G_1$ containing $e$. Since $C$ consists of the edges incident with the contraction vertex of $G_1$, $M\cap C=\{e\}$.
	
	If $k\geq5$, then every spoke of $G_1$ is removable.
	Since $e$ is nonremovable in $G_1$, $e$ is a rim edge.
	Let $e'$ be the unique spoke of $M$. Then $e'\notin C$. By Lemma \ref{lem:re_also_re}, $e'$ is removable in $G$. If $k=3$, then $W_3=K_4$ and $M=\{e,e'\}$ for an edge $e'\notin C$. Since $e$ is nonremovable in $G_1$, it is not a multiple edge. If $e'$ is a multiple edge, then $e'$ is removable in $G_1$, and hence in $G$ by Lemma \ref{lem:re_also_re}. Otherwise, neither edge of $M$ is a multiple edge. Thus $M$ is a removable doubleton of $G_1$.
	Since $G_1$ is a brick, $M\cap C=\{e\}$, and $e$ is removable in $G_2$, $e'$ is removable in $G$ by Lemma \ref{thm:removable doubleton}. Thus the claim holds. By Corollary \ref{lem:MC} (2), $e$ is solitary in $G_2$, a contradiction.
	Therefore, every removable edge of $G_2$ is solitary. 
\end{proof}

Next, we present some properties of solid bricks that will be used in the proof of the main result.
\begin{lem}{\rm\cite{Lucchesi2024}}\label{lem:solid-nonremovable}
	If $G$ is a solid brick with $|V(G)|\ge 6$, then every vertex of $G$ is incident with at most two nonremovable edges of $G$.
\end{lem}

\begin{lem}[{\rm\cite{DL2026}}]\label{lem:solitary-common-neighbors}
	Let $uv$ be a solitary edge of a solid brick $G$.
	Then $G-u-v$ has two distinct vertices of degree one that are adjacent to both $u$ and $v$ in $G$.
\end{lem}

\begin{lem}[{\rm\cite{DL2026}}]
	\label{lem:solitary-cubic-neighbors}
	Let $G$ be a simple solid brick, and let $ux$ be a solitary edge with $d(x)=3$. Then the other two neighbors of $x$ have degree three and are adjacent to $u$.
\end{lem}

\begin{pro}[{\rm\cite{LuX24}}]\label{pro:uvnonremovable}
	Let $G$ be a simple matching covered graph. Assume that $N(u)=\{u_1,u_2,u_3\}$ and $G[\{u,u_2,u_3\}]$ is a triangle. Then $uu_1$ is not removable.
\end{pro}

\begin{pro}[{\rm\cite{Lucchesi2024}}]\label{pro:Solid Re}
	Let $G$ be a solid brick and $v_0\in V(G)$. Either $G$ is an odd wheel with $v_0$ as a hub, possibly with multiple edges incident with $v_0$, or $G$ has two removable edges not incident with $v_0$.
\end{pro}

\begin{lem}
	\label{lem:solid-brick-odd-wheel}
	Let $G$ be a simple solid brick. If $G$ contains a vertex $v_0$ such that all removable edges in $E(G)\setminus\partial(v_0)$ are solitary, then $G$ is an odd wheel.
\end{lem}
\begin{proof}
	Since $G$ is a brick, it is 3-connected and has even order.
	If $|V(G)|=4$, then $G=K_4$, which is an odd wheel.
	Assume that $|V(G)|\geq6$.
	If $G$ is a wheel with hub $v_0$, then its rim is odd.
	Assume that $G$ is not such a wheel.
	By Proposition~\ref{pro:Solid Re}, $G$ has a removable edge
	$uv\notin\partial(v_0)$.
	By hypothesis, $uv$ is solitary. Since $uv\notin\partial(v_0)$,
	we have $u,v\neq v_0$.
	
	By Lemma~\ref{lem:solitary-common-neighbors}, $G-u-v$ has two
	distinct vertices $x,y$ of degree one that are adjacent to both
	$u$ and $v$ in $G$. Since $G$ is simple, $d(x)=d(y)=3$.
	Since $x$ and $y$ are distinct, at least one differs from $v_0$.
	Without loss of generality, assume that $x\neq v_0$.
	Let $N(x)=\{u,v,z\}$.
	Since $uvx$ is a triangle, $xz$ is not removable by
	Proposition~\ref{pro:uvnonremovable}.
	By Lemma~\ref{lem:solid-nonremovable}, at least one of $ux$ and
	$vx$ is removable. Since both edges lie outside $\partial(v_0)$,
	at least one is solitary.
	By the symmetry of $u$ and $v$, assume that $ux$ is solitary.
	By Lemma~\ref{lem:solitary-cubic-neighbors}, $d(v)=d(z)=3$
	and $uz\in E(G)$.
	If $z=y$, then $v,x,y$ form a triangle. Each of these vertices
	has degree three and is adjacent to $u$. Thus none has a neighbor
	outside $\{u,v,x,y\}$, and this triangle is a component of $G-u$.
	Since $G-u$ is connected, $|V(G)|=4$, a contradiction.
	Hence $y,v,x,z$ are distinct and form a path. Each of these vertices
	has degree three in $G$ and is adjacent to $u$.
	
	Choose a longest path $P$ in $G-u$ whose vertices all have degree
	three in $G$ and are adjacent to $u$.
	Since $y,v,x,z$ form such a path, $P$ has at least four vertices.
	Since the ends of $P$ are distinct, at least one differs from $v_0$.
	Choose an end $w$ of $P$ distinct from $v_0$. Let $w,r,s$ be
	the first three vertices of $P$ in order from $w$. Let $N(w)=\{u,r,t\}$.
	Since $d(r)=3$, we have $N(r)=\{u,w,s\}$.
	By Proposition~\ref{pro:uvnonremovable}, neither $rw$ nor $wt$
	is removable, as $urs$ and $urw$ are triangles.
	Since $d(w)=3$, the edge $uw$ is removable by
	Lemma~\ref{lem:solid-nonremovable}.
	Since $u,w\neq v_0$, the edge $uw$ is solitary by hypothesis.
	By Lemma~\ref{lem:solitary-cubic-neighbors}, $d(t)=3$
	and $ut\in E(G)$.
	
	If $t\notin V(P)$, then $wt$ extends $P$ to a longer path
	whose vertices all have degree three in $G$ and are adjacent to $u$.
	This contradicts the choice of $P$. Hence $t\in V(P)$.
	Since every internal vertex of $P$ has degree three in $G$ and is adjacent to $u$, its only other neighbors are its two neighbors on $P$.
	As $t\neq w,r$, the edge $wt$ is not an edge of $P$.
	If $t$ were internal, $w$ would be a fourth neighbor of $t$,
	a contradiction. Thus $t$ is the other end of $P$.
	The edges of $P$ together with $wt$ form a cycle.
	Every vertex of this cycle has degree three in $G$ and is adjacent to $u$. Hence each vertex of the cycle has only its two cycle neighbors
	in $G-u$. Thus the cycle is a component of $G-u$.
	Since $G-u$ is connected, this cycle contains every vertex of $G-u$.
	Hence $G$ is a wheel with hub $u$.
	Since $|V(G)|$ is even, its rim is odd.
\end{proof}

By \cite[Theorem~7.8]{Lucchesi2024}, a brick $G$ is nonsolid if and only if it contains two vertex-disjoint odd cycles $C_1$ and $C_2$ such that $G-(V(C_1)\cup V(C_2))$ has a perfect matching.

\begin{pro}\label{pro:mul-solid}
	Let $G$ be a loopless graph, and let $H$ be its underlying simple graph. Then $G$ is a solid brick if and only if $H$ is a solid brick.
\end{pro}

\begin{proof}
	For every $S\subseteq V(G)$, the graph $H-S$ is the underlying simple graph of $G-S$. Thus $G-S$ and $H-S$ have the same components, and one has a perfect matching if and only if the
	other does. Hence $G$ is a brick if and only if $H$ is a brick.
	Since $G$ and $H$ also have odd cycles on the same vertex sets,
	the characterization above implies that $G$ is a solid brick
	if and only if $H$ is a solid brick.
\end{proof}

\section{Proof of Theorem \ref{main} and a Family of Graphs}\label{3}

\subsection{Proof of Theorem \ref{main}}\label{3.1}

Let $G$ be a simple nonsolid brick in which every removable edge is solitary.
By Lemma \ref{lem:Robust-solid}, $G$ has a robust cut $\partial(X)$ such that there exist subsets $X'$ of $X$ and $X''$ of $\overline{X}$ such that $G/(\overline{X'}\rightarrow \overline{x'})$ is a solid brick $G_1$, $G/(\overline{X''}\rightarrow \overline{x''})$ is a brick $G_2$ and the graph $H$, obtained from $G$ by contracting $X'$ and $X''$ to single vertices $x'$ and $x''$, respectively, is bipartite and matching covered, where $x'$ and $x''$ lie in different color classes of $H$.
Since $G/X'$ is obtained by splicing $H$ at $x''$ with $G_2$ at $\overline{x''}$, it is matching covered by Proposition \ref{pro:MC_IS_MC}. Thus $\partial(X')$ is a separating cut of $G$. Since $G_1$ and $G_2$ are bricks, $|X'|,|X''|\geq3$. Hence $\partial(X')$ is nontrivial. Since $G$ is a brick, it is not tight.

\begin{claim}
	\label{clm:outside-solitary}
	Every removable edge of $G_1$ not incident with $\overline{x'}$ is solitary in $G_1$.
\end{claim}
\begin{proof}
	By Lemma \ref{lem:re_also_re}, every such edge is removable in $G$.
	By hypothesis, it is solitary in $G$.
	By Proposition \ref{pro:solitary}, it is solitary in $G_1$.
\end{proof}

Let $W$ be the underlying simple graph of $G_1$, regarded as a spanning subgraph of $G_1$. By Proposition \ref{pro:mul-solid}, $W$ is a solid brick. Since $G$ is simple, every multiple edge of $G_1$ is incident with $\overline{x'}$. Thus, if $e$ is a removable edge of $W$ not incident with $\overline{x'}$, then $G_1-e$ is obtained from $W-e$ by adding multiple edges. Hence $e$ is removable in $G_1$ and is solitary in $G_1$ by Claim \ref{clm:outside-solitary}.
Since $W$ is matching covered and every perfect matching of $W$ is also one of $G_1$, $e$ is solitary in $W$.
By Lemma \ref{lem:solid-brick-odd-wheel}, $W$ is an odd wheel.

If $|V(W)|=4$, take $\overline{x'}$ as the hub of $W=K_4$.
Then every multiple edge of $G_1$ is a spoke.
Assume that $|V(W)|\geq6$. Suppose that a rim edge $e$ of $W$ is a multiple edge of $G_1$. Note that $e$ is incident with $\overline{x'}$. Let $M$ be a perfect matching of $W$ containing
$e$, and let $f$ be its unique spoke. Since $e\in M$, the edge $f$ is not incident with $\overline{x'}$. 
Recall that the spoke $f$ is removable in $W$ and hence in $G_1$.
By Claim~\ref{clm:outside-solitary}, $M$ is the unique perfect matching of $G_1$ containing $f$. However, replacing $e$ in $M$ with another edge of $G_1$ joining the same two vertices yields a second such matching, a contradiction.
Therefore $G_1$ is an odd wheel, possibly with multiple spokes.
By Lemma \ref{lem:G2re-is-so}, every removable edge of $G/X'$ is solitary.
\begin{claim}\label{clm:H-two-vertices}
	$V(H)=\{x',x''\}$.
\end{claim}
\begin{proof}
	Suppose that $|V(H)|\geq4$.
	Since $H$ is matching covered, $x'$ has at least two distinct neighbors.
	Choose $e\in\partial_H(x')\setminus E_H[x',x'']$.
	By Lemma \ref{lem:2-bi-removable}, $e$ is removable in $H$.
	The two $\partial_{G/X'}(X'')$-contractions of $G/X'$ are $H$ and $G_2$.
	Thus $\partial_{G/X'}(X'')$ is a separating cut.
	Since $e\notin E(G_2)$, the edge $e$ is removable in $G/X'$ by Lemma \ref{lem:re_also_re}.
	By Lemma \ref{lem:H-not-solitary}, $e$ is not solitary in $H$.
	By Proposition \ref{pro:solitary}, $e$ is not solitary in $G/X'$.
	This contradicts the preceding conclusion that every removable edge of $G/X'$ is solitary.
\end{proof}

By Claim~\ref{clm:H-two-vertices}, $G=(G_1(\overline{x'})\odot G_2(\overline{x''}))_\theta$
for the bijection $\theta$ induced by the edges of $\partial(X')$.
Moreover, $G_2=G/X'$, so every removable edge of $G_2$ is solitary.
This proves Theorem \ref{main}.
\hfill\qedsymbol

Let $G$ be a simple nonsolid brick in which every removable edge is solitary. We prove by induction on $|V(G)|$ that $G$ can be obtained by repeated splicing of odd wheels, up to multiple edges. By Theorem \ref{main}, $G$ is a splicing of an odd wheel, up to multiple edges, and a smaller brick $H$ in which every removable edge is solitary. Let $H_0$ be the underlying simple graph of $H$, regarded as a spanning subgraph of $H$. Then $H_0$ is also a brick. Let $e$ be a removable edge of $H_0$. Since the underlying simple graph of $H-e$ is either $H_0$ or $H_0-e$, both of which are matching covered, $e$ is removable in $H$.
Thus $e$ is solitary in $H$. Since $H_0$ is matching covered and every perfect matching of $H_0$ is one of $H$, $e$ is solitary in $H_0$. 

If $H_0$ is solid, then $H_0$ is an odd wheel by Lemma \ref{lem:solid-brick-odd-wheel}, so $H$ is an odd wheel up to multiple edges. Now suppose that $H_0$ is nonsolid. Since $|V(H_0)|<|V(G)|$, the induction hypothesis
applies to $H_0$. 
Let $x$ be the contraction vertex of $H$. Since $G$ is simple, every multiple edge of $H$ is incident with $x$. 
To restore a parallel edge of $xy$, let $C$ be the cut of the last splicing in the representation of $H_0$. If $xy\notin C$, duplicate $xy$ in the graph containing it. If $xy\in C$, duplicate the corresponding edge in each graph and pair the two new edges in the splicing bijection. If either graph in which an edge is to be duplicated is itself obtained by splicing, apply the same argument to its last splicing.
Thus all the required parallel edges can be restored by adding parallel edges to the odd wheels in the representation of $H_0$.

In either case, the original splicing bijection yields $G$ by splicing $H$ with the odd wheel given by Theorem \ref{main}. Hence $G$ can be obtained by repeated splicing of odd wheels, up to multiple edges.

\subsection{A Family of Graphs}\label{family-graphs}
By Theorem~\ref{main}, it is natural to ask whether every simple nonsolid brick in which every removable edge is solitary can be obtained by splicing a brick and a graph whose underlying simple graph is $K_4$.
We construct an infinite family showing that the answer is negative.

Fix odd integers $\ell,k\geq5$.
Let $W'_\ell$ and $W_k$ be odd wheels with hubs $h'_\ell$ and $h_k$,
and rim vertices $w'_1,\ldots,w'_\ell$ and $w_1,\ldots,w_k$, respectively.
Let $L:=W'_\ell-w'_\ell$ and $R:=W_k-w_k$ be the terminal gadgets.
For each $t\geq0$, let $B_t$ have vertex set
$\{v_1^t,v_2^t,v_3^t,v_4^t\}$ and edge set
$\{v_1^tv_2^t,v_3^tv_4^t\}$.
For each $t\geq1$, introduce two connector vertices $w_2^t,w_3^t$.

For each $i\geq0$, define $G_i$ by
\[
V(G_i)=V(L)\cup V(R)\cup
\bigcup_{t=0}^{i}V(B_t)\cup
\bigcup_{t=1}^{i}\{w_2^t,w_3^t\}.
\]
Write $w_3^0=h_k$ and $w_3^{i+1}=h'_\ell$.
The edge set of $G_i$ consists of the following six families:
\begin{enumerate}
	\item All edges of $L$ and $R$;
	\item All edges of $B_t$, for $0\leq t\leq i$;
	\item For each $0\leq t\leq i$, the edge $w_3^tw_3^{t+1}$ and the edges
	$w_3^tv_j^t,w_3^{t+1}v_j^t$, for $1\leq j\leq4$;
	\item The edges $w_2^tw_3^t$, for $1\leq t\leq i$;
	\item The edges of the upper boundary path
	$w'_1v_1^iv_1^{i-1}\cdots v_1^0w_1$;
	\item The lower boundary edges $w'_{\ell-1}v_4^i$, $v_4^0w_{k-1}$,
	and $v_4^tw_2^t,w_2^tv_4^{t-1}$ for $1\leq t\leq i$.
	These form a path, which is $w'_{\ell-1}v_4^0w_{k-1}$ when $i=0$.
\end{enumerate}
Thus $G_i$ is a chain of the layers $B_i,B_{i-1},\ldots,B_0$ between
$L$ and $R$. It is simple and connected, and
$|V(G_i)|=\ell+k+6i+4$.
Hence $\mathcal G=\{G_i:i\geq0\}$ is an infinite family.	
The graphs $G_1$ and $G_2$ are depicted in Figure \ref{fig:Gi-G1} for illustration only.

\begin{figure}[!h]
	\centering
	\begin{minipage}[t]{0.3\textwidth}
	\centering
	\includegraphics[totalheight=3cm]{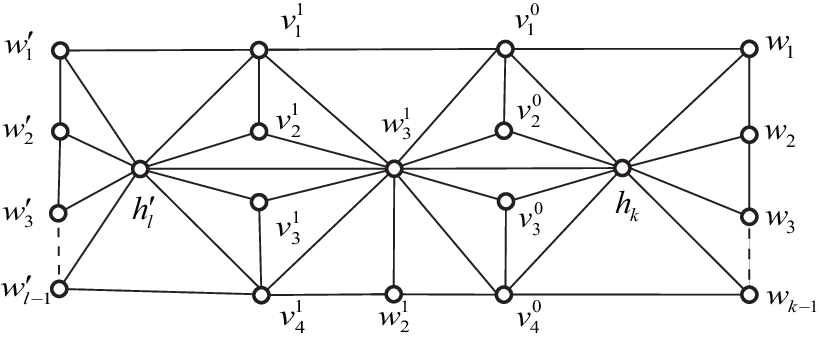}
	\end{minipage}

\centering	
	\begin{minipage}[t]{0.45\textwidth}
		\centering
		\includegraphics[totalheight=3cm]{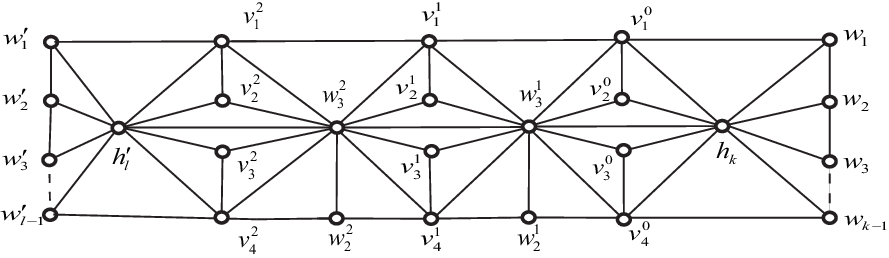}
	\end{minipage}
	\caption{The graphs $G_1$ and $G_2$.}
	\label{fig:Gi-G1}
\end{figure}

\begin{thm}\label{main1}
	Every graph in $\mathcal G$ is a nonsolid brick in which every removable
	edge is solitary. Moreover, no graph in $\mathcal G$ can be obtained by
	splicing a brick and a graph whose underlying simple graph is $K_4$.
\end{thm}
\begin{proof}
	Fix $i\geq0$. We first prove that $G_i$ is a nonsolid brick.
	Contract $V(L)$ to a vertex $x$.
	Contracting $V(G_i)\setminus V(L)$ to one vertex gives $W'_\ell$
	with the spoke $h'_\ell w'_\ell$ of multiplicity five.
	We show that $G_i/V(L)$ is a brick by a sequence of triangle contractions.
	
	Starting with $j=i$, contract $\{x,v_1^j,v_2^j\}$ and then
	$\{x,v_3^j,v_4^j\}$, using $x$ for the new vertex after each contraction.
	If $j\geq1$, also contract $\{x,w_2^j,w_3^j\}$ and continue with layer $j-1$.
	Before layer $j$ is processed, the neighbors of $x$ are
	$v_1^j,v_2^j,v_3^j,v_4^j,w_3^j$, with respective edge multiplicities
	$2,1,1,2,1$.
	This holds after contracting $L$.
	For $j\geq1$, the first two contractions leave $x$ adjacent only to
	$v_1^{j-1},w_2^j,w_3^j$, with multiplicities $1,1,5$.
	The third contraction gives two edges to each of $v_1^{j-1},v_4^{j-1}$
	and one edge to each of $v_2^{j-1},v_3^{j-1},w_3^{j-1}$.
	Thus the same description holds at the next layer.
	At layer $0$, the first two contractions instead leave edges to
	$w_1,w_{k-1},h_k$ with multiplicities $1,1,5$.
	The final graph is therefore $W_k$ with the spoke $h_kw_k$ of multiplicity five.
	
	Each contracted set is the vertex set of a triangle whose cut contains
	three independent edges.
	For $j\geq1$, the first cut contains
	$xv_3^j,v_1^jv_1^{j-1},v_2^jw_3^j$.
	After the first contraction, the second cut contains
	$xv_1^{j-1},v_3^jw_3^j,v_4^jw_2^j$.
	After the second contraction, the third cut contains
	$xv_1^{j-1},w_2^jv_4^{j-1},w_3^jv_2^{j-1}$.
	At layer $0$, the corresponding first two triples are
	$xv_3^0,v_1^0w_1,v_2^0h_k$ and
	$xw_1,v_3^0h_k,v_4^0w_{k-1}$.
	For each triangle, contract all vertices outside it in the current graph
	to one vertex. The resulting graph has underlying simple graph $K_4$
	and is therefore a brick.
	
	Now reverse these contractions.
	At each step, splice the current brick with the corresponding graph whose underlying simple graph is \(K_4\). The splicing vertices are the two contraction vertices. For each edge of the cut, the bijection pairs the corresponding edge in one contraction with the corresponding edge in the other.
	By Proposition~\ref{pro:MC_IS_MC}, the resulting graph is matching covered.
	The splicing cut contains three independent edges, so no two vertices cover it. By Theorem~\ref{thm:2-cut}, the resulting graph is a brick.
	Starting from the final wheel proves that $G_i/V(L)$ is a brick.
	Finally, $\partial_{G_i}(V(L))$ contains the three independent edges
	$w'_1v_1^i,w'_{\ell-1}v_4^i,h'_\ell v_2^i$.
	Its two contractions are bricks, so the same argument proves that $G_i$
	is a brick. This cut is nontrivial and separating.
	Since a brick has no nontrivial tight cut, it is not tight.
	Thus $G_i$ is nonsolid.
	
	Let $U=\{w_3^t:0\leq t\leq i+1\}$, and call its vertices hubs.
	Define
	\[
	S_i=\left(\bigcup_{x\in U}\partial_{G_i}(x)\right)
	\setminus\left(\bigcup_{t=0}^{i}
	\bigl(\partial_{G_i}(v_2^t)\cup\partial_{G_i}(v_3^t)\bigr)\right).
	\]
	This set consists of the terminal spokes, the edges \(w_2^t w_3^t\) for \(1 \le t \le i\), and five edges at each layer: the edge joining its two adjacent hubs and the four edges joining those hubs to \(v_1^j, v_4^j\).
	We now prove the following two claims.
	\\[7pt]\noindent
	\textbf{Claim A.} Every removable edge of $G_i$ is solitary.
	
	\begin{proof}
	Set $L_{i+1}=L$ and $R_0=R$. For $1\le t\le i$, let $L_t$
	be induced by $L$, the layers $B_t,\ldots,B_i$, and the connector
	pairs with indices $t,\ldots,i$. Let $R_t$ be induced by $R$,
	the layers $B_0,\ldots,B_{t-1}$, and the connector pairs with
	indices $1,\ldots,t$. The boundary vertices of $L_t$ are
	$v_1^t,w_2^t,w_3^t$, and those of $R_t$ are
	$v_1^{t-1},w_2^t,w_3^t$. We first show that deleting $w_3^t$
	or all three boundary vertices from either graph leaves a graph
	with a unique perfect matching. Since $L_{i+1}-w_3^{i+1}$ and
	$R_0-w_3^0$ are paths of even order, each has a unique perfect
	matching. In $L_t-w_3^t$, the vertex $w_2^t$ must be matched to
	$v_4^t$, which in turn forces $v_3^tw_3^{t+1}$ and $v_1^tv_2^t$.
	In $L_t-\{v_1^t,w_2^t,w_3^t\}$, the vertex $v_2^t$ must be
	matched to $w_3^{t+1}$, which forces $v_3^tv_4^t$. In either
	case, deleting the ends of the forced edges leaves
	$L_{t+1}-w_3^{t+1}$. Thus the assertion for $L_t$ follows by
	induction from $t=i$ down to $t=1$. Reversing the chain gives
	the assertion for $R_t$.
	
	Let $e\in S_i$. Since $G_i$ is matching covered, it has a perfect matching $M$ containing $e$. First consider a layer $B_j$, and let $a=w_3^j$ and $b=w_3^{j+1}$. If $e=ab$, then $M$ contains both edges of $B_j$. If $e$ joins $v_1^j$ to one of $a,b$, then $M$ matches $v_2^j$ to the other and contains $v_3^jv_4^j$. Similarly, if $e$ joins $v_4^j$ to one of $a,b$, then $M$ matches $v_3^j$ to the other and contains $v_1^jv_2^j$.
	In each case, deleting the ends of these three matching edges leaves $L_{j+1}-w_3^{j+1}$ and $R_j-w_3^j$, each with a unique perfect matching. Next, let $e=w_2^tw_3^t$. In $G_i-w_2^t-w_3^t$, the edge $v_1^tv_1^{t-1}$ is a bridge whose deletion leaves two components of odd order, since each consists of one terminal gadget of odd order, layers of four vertices, and connector pairs.
	Thus $v_1^tv_1^{t-1}\in M$. Deleting the ends of these two matching edges leaves $L_t$ and $R_t$ with their respective three boundary vertices deleted, so again $M$ is unique. Finally, let $e=h'_\ell w'_r$. After deleting the ends of $e$, the remaining portion of the left rim path of odd order has only one edge to the rest of the graph. Thus $M$ contains $w'_1v_1^i$ if $r$ is even and $w'_{\ell-1}v_4^i$ if $r$ is odd. In the first case, $M$ must also contain $v_2^iw_3^i$ and $v_3^iv_4^i$, and in the second it must contain $v_3^iw_3^i$ and $v_1^iv_2^i$. What remains consists of even rim paths and $R_i-w_3^i$, each with a unique perfect matching. The same argument applies to the right terminal spokes. Hence every edge of $S_i$ is solitary.
	
	For $e\in S_i$, let $M_e$ be the unique perfect matching containing $e$. By Proposition \ref{pro:Me-nonre}, every edge of $M_e\setminus\{e\}$ is nonremovable. It therefore suffices to show that every edge outside $S_i$ belongs to some $M_e$.
	Both edges of $B_j$ belong to $M_{w_3^jw_3^{j+1}}$. By the argument for $B_j$, each of the four edges joining
	$\{v_2^j,v_3^j\}$ to $\{w_3^j,w_3^{j+1}\}$ belongs to $M_e$ for some $e\in S_i$.
	Each internal upper boundary edge $v_1^tv_1^{t-1}$ belongs to $M_{w_2^tw_3^t}$, and the two terminal upper boundary edges $w'_1v_1^i$ and $v_1^0w_1$ belong to $M_{h'_\ell w'_2}$ and $M_{h_kw_2}$, respectively. Let $M_L=M_{h'_\ell w'_1}$ and $M_R=M_{h_kw_1}$.
	By the terminal spoke argument and the induction for $R_t$, $M_L$ contains $w'_{\ell-1}v_4^i$ and every $w_2^tv_4^{t-1}$.
	Similarly, $M_R$ contains $v_4^0w_{k-1}$ and every $v_4^tw_2^t$.
	Thus the two matchings cover the lower boundary. On the left rim, $M_L$ contains $w'_rw'_{r+1}$ for even $r$, and $M_R$ contains these edges for odd $r$. On the right rim, their roles are reversed. Hence every edge outside $S_i$ is nonremovable.
	Since every edge of $S_i$ is solitary, every removable edge of $G_i$ is solitary.
	\end{proof}

	\noindent
	\textbf{Claim B.} The graph \(G_i\) cannot be obtained by splicing a brick and a graph whose underlying simple graph is \(K_4\).
	
	\begin{proof}
	Since \(G_i - U\) is a cycle of length at least ten with leaves attached and \(G_i[U]\) is a path, every triangle of \(G_i\) contains exactly one or two vertices of \(U\). Suppose, to the contrary, that \(G_i\) can be obtained by splicing a brick and a graph whose underlying simple graph is \(K_4\). By the definition of splicing, there is a triangle \(T\) in \(G_i\) such that \(G_i/V(T)\) is a brick. Let \(x\) be the contracted vertex. If \(V(T) = \{w_3^t, v_1^t, v_2^t\}\) or \(V(T) = \{w_3^t, v_1^{t-1}, v_2^{t-1}\}\) for some \(1 \le t \le i\), then deleting \(x\) and \(w_2^t\) from \(G_i/V(T)\) leaves \(L\) and \(R\) in distinct components. Thus \(\{x, w_2^t\}\) is a 2-vertex cut. Otherwise, by the construction of \(G_i\), the vertex \(x\) has a neighbor \(u\) with exactly two distinct neighbors in \(G_i/V(T)\). These two neighbors form a 2-vertex cut of \(G_i/V(T)\). Thus \(G_i/V(T)\) is not 3-connected in either case, a contradiction. Therefore, the  claim holds.
	\end{proof}
	
	By the nonsolidity of \(G_i\) and Claims A and B, Theorem \ref{main1} follows.
\end{proof}

\subsection{Splicings of Two Odd Wheels}\label{two-odd}

As shown in Section~\ref{3.1}, every simple nonsolid brick in which every removable edge is solitary can be obtained by successively splicing odd wheels (up to multiple edges). We now consider a simple brick $G$ obtained by splicing two graphs whose underlying simple graphs are odd wheels. The following proposition shows that if every removable edge of $G$ is solitary, then one of these wheels must be $W_3$, i.e., $K_4$.

\begin{pro}\label{pro:2-wheel}
Let $H_s$ and $H_t$ have underlying simple graphs $W_s$ and
$W_t$, respectively, where $s\geq t\geq3$ are odd integers.
Let $G\in H_s(\overline{x})\odot H_t(x)$ be a simple brick in which every removable edge is solitary. Then $t=3$. Moreover, if $s\geq5$, then $\overline{x}$ is not the hub of $W_s$.	
\end{pro}	
\begin{proof}
	Since \(H_s-\overline{x}\) and \(H_t-x\) are subgraphs of the simple graph \(G\), every multiple edge of \(H_s\) is incident with \(\overline{x}\), and every multiple edge of \(H_t\) is incident with \(x\). Let \(C:=\partial_G(V(H_s)\setminus\{\overline{x}\})\). As $C$ is a nontrivial cut of the brick $G$, it is not tight.
	Since $V(H_s)\setminus\{\overline{x}\}$ has odd cardinality, there is a perfect matching $M$ of $G$ such that $|M\cap C|$ is odd and at least three. 
	
	Suppose, to the contrary, that \(t\geq5\). Let \(h_1\) and \(h_2\) be the hubs of \(W_s\) and \(W_t\), respectively. Every spoke of \(H_s\) or \(H_t\) is removable. Moreover, no rim edge of \(H_s\) or \(H_t\) is solitary. We claim that \(M\cap C\subseteq\partial_{H_s}(h_1)\cup\partial_{H_t}(h_2)\).
	Otherwise, some \(e\in M\cap C\) is a rim edge in both \(H_s\) and \(H_t\). Let \(N\) be a perfect matching of \(H_s\) containing \(e\), and let \(e'\) be its unique spoke. Since \(N\cap C=\{e\}\), we have \(e'\notin C\). Since \(e'\) is removable in \(H_s\), it is removable in \(G\) by Lemma \ref{lem:re_also_re}. By Corollary \ref{lem:MC} (2), \(e\) is solitary in \(H_t\), a contradiction.
	
	Consider either \(H_s\) or \(H_t\). If its splicing vertex is the hub, every edge of $M\cap C$
	corresponds to a spoke in that graph. If its splicing vertex is on the rim, its three distinct neighbors are the hub and two rim vertices.  Since $M$ is a matching and $|M\cap C|\geq3$, exactly one edge of $M\cap C$ is incident with each of these three neighbors in $G$. Only the edge incident with the hub corresponds to a spoke. Thus both \(|M\cap C\cap\partial_{H_s}(h_1)|\) and \(|M\cap C\cap\partial_{H_t}(h_2)|\) are odd. By the claim, the union of these two edge sets is \(M\cap C\). By inclusion-exclusion, their intersection has odd cardinality. Hence some \(f\in M\cap C\) is a spoke in both \(H_s\) and \(H_t\). By Lemma \ref{lem:re_also_re}, \(f\) is removable in \(G\), contradicting Corollary \ref{lem:MC} (1). Therefore \(t=3\).
	
	Next, consider the case where \(s\geq5\) and \(t=3\). Suppose that \(\overline{x}\) is the hub of \(W_s\). Since \(x\) has exactly three distinct neighbors in \(H_3\) and \(|M\cap C|\geq3\), the three edges of \(M\cap C\) meet all three neighbors. Since \(d_{H_3}(x)=d_{H_s}(\overline{x})\geq5\), some edge \(f'\in M\cap C\) is a multiple edge of \(H_3\) and hence is removable in \(H_3\). Since every edge of \(\partial_{H_s}(\overline{x})\) is removable in \(H_s\), \(f'\) is removable in \(G\) by Lemma \ref{lem:re_also_re}, contradicting Corollary \ref{lem:MC} (1). Therefore, the result follows.	
\end{proof}	

\section{Bricks in Which the Set of Removable Edges Coincides with the Set of Solitary Edges}\label{sec:R=S} 
Let \(R_G\) and \(S_G\) denote the sets of removable and solitary edges of \(G\), respectively.
We omit the subscript $G$ if $G$ is understood.
In this section, we consider the structure of bricks satisfying the condition $R = S$. 
Note that every odd wheel $W_k$ ($k \ge 5$), possibly with multiple edges incident to its hub vertex, is a brick satisfying $R = S$. 
\begin{thm}\label{R=S}
	Let $G\in G_1\odot G_2$, where
	$G_i$ is a brick with $R_{G_i}=S_{G_i}$ for each $i\in\{1,2\}$. Then
	$G$ is a brick with $R=S$ if and only if $G$ is a brick with $R\subseteq S$.
\end{thm}
\begin{proof}
	We only   prove the sufficiency while the necessity
	is trivial.
	Suppose that there exists an edge $e$ that is nonremovable and solitary in $G$.
	Then $e$ is nonremovable in at least one of $G_1$ and $G_2$ by Lemma \ref{lem:re_also_re}.
	Without loss of generality, assume that  $e$ is nonremovable in $G_1$.
	As $G_1$ is a brick with $R_{G_1}=S_{G_1}$, $e$ is not solitary in $G_1$. So $e$ is not solitary in $G$ by Proposition \ref{pro:solitary}, a contradiction. 
	Thus, the result holds.
\end{proof}

We remark that the conclusion of Theorem \ref{R=S} remains valid even if exactly one of $G_1$ and $G_2$ fails to satisfy $R=S$.
This is formally established by the following proposition.

\begin{pro}\label{R=S-K4}
	Let $G_1$ be a brick with $R_{G_1}=S_{G_1}$ and $G_2$ be a $K_4$ {\rm(}possibly with multiple edges incident to the same vertex{\rm)} with $R_{G_2}\subseteq S_{G_2}$. Assume that $G\in G_1\odot G_2$. Then
	$G$ is a brick with $R=S$ if and only if $G$ is a brick with $R\subseteq S$.
\end{pro}
\begin{proof}
	We only prove the sufficiency while the necessity
	is trivial. Let $C:=\partial_G(V(G_1)\setminus \{\overline{x}\})$, where $\overline{x}$ is the splicing vertex of $G_1$.
	Suppose that there exists an edge $e$ that is nonremovable and solitary in $G$.
	By Lemma \ref{lem:re_also_re},
	$e$ is nonremovable in at least one of $G_1$ and $G_2$.
	If $e$ is nonremovable in $G_1$, the argument is similar to that in Theorem \ref{R=S}.
	Assume that $e$ is a nonremovable edge of $G_2$. Then $e$ is not a multiple edge of $G_2$. 
	If $e\in C$, then $e$ belongs to a perfect matching $M$ of $G$ with $|M\cap C|=3$, since $C$ is not a tight cut of $G$.
	Since $C$ is a separating cut of $G$, $e$ belongs to a perfect matching $M_e$ of $G$ with $|M_{e}\cap C|=1$ by Proposition \ref{pro:separating}.
	Hence $e$ belongs to two distinct perfect matchings of $G$, contradicting the hypothesis that $e$ is solitary in $G$.
	Thus, $e\notin C$.
	Since the underlying graph of $G_2$ is isomorphic to $K_4$, and $e$ is a solitary edge of $G_2$ by Proposition \ref{pro:solitary}, $e$ belongs to a removable doubleton of $G_2$, say $D:=\{e,f\}$, and $f$ is not a multiple edge of $G_2$. Note that $f\in C$.
	We now distinguish two cases.
	If $f\in R_{G_1}$, then $e\in R$ by Lemma \ref{thm:removable doubleton}, contradicting the hypothesis that $e$ is nonremovable in $G$.
	If $f\notin R_{G_1}$, then $f\notin S_{G_1}$. So $G_1$ contains two distinct perfect matchings $M_1$ and $M_2$ both containing $f$.
	Hence $M_1\cup \{e\}$ and $M_2\cup \{e\}$ are two distinct perfect matchings of $G$, contradicting the hypothesis that $e$ is solitary in $G$.
		Therefore, the result holds.	
	\end{proof}

We remark that the condition that both $G_1$ and $G_2$ satisfy $R = S$ is not necessary for Theorem \ref{R=S}, and the conclusion may fail if neither $G_1$ nor $G_2$ satisfies $R = S$. We illustrate this with two examples.
First, there exist bricks $G_1$ and $G_2$ with $R_{G_i} \subseteq S_{G_i}$ for each $i\in\{1,2\}$ such that $(G_1 \odot G_2)_{\theta}$  
is a brick with $R_G \subseteq S_G$ but $R_G \ne S_G$ (see Fig. \ref{fig:G1}).
Second, there exists a brick $H$ with $R_H = S_H$ that can be decomposed as  $(H_1 \odot H_2)_{\theta}$,
where $R_{H_i} \neq S_{H_i}$ for each $i\in \{1, 2\}$ (see Fig. \ref{fig:G}). In both figures, bold edges represent edges that are both removable and solitary, while red edges represent edges that are solitary but nonremovable.

\begin{figure}[!h]
	\centering
	\includegraphics[totalheight=3cm]{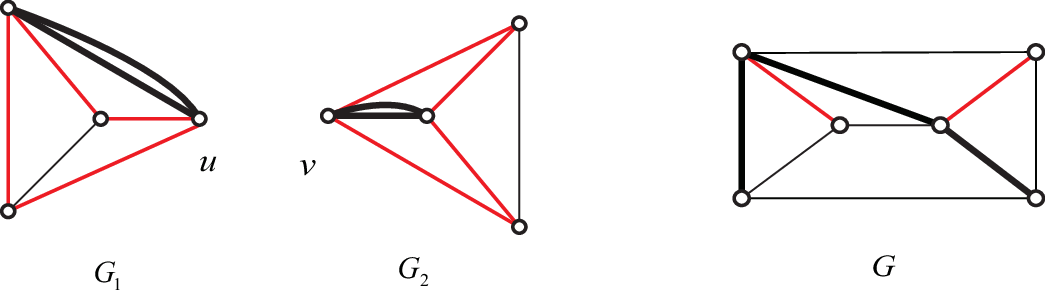}
	\caption{The graphs $G_1$, $G_2$ and $G=(G_1(u)\odot G_2(v))_{\theta}$.}
	\label{fig:G1}
\end{figure}

\begin{figure}[!h]
	\centering
	\includegraphics[totalheight=3cm]{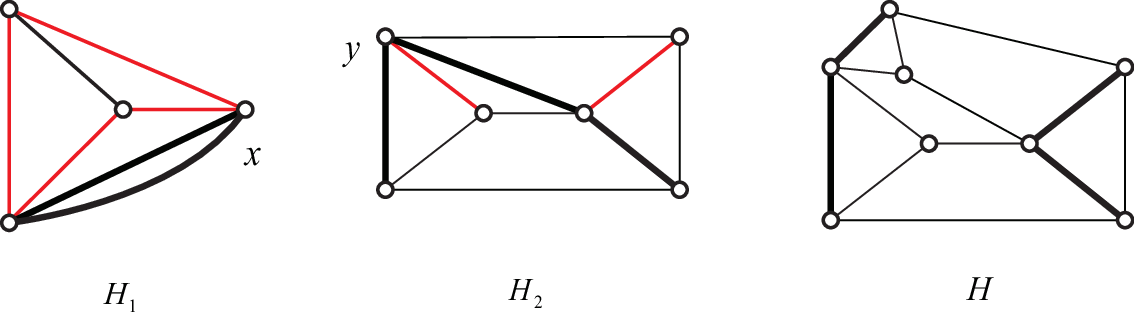}
	\caption{The graphs $H_1$, $H_2$ and $H=(H_1(x)\odot H_2(y))_{\theta}$.}
	\label{fig:G}
\end{figure}

\section{Declaration of AI use}

The authors used AI only to improve the language and readability of the complete manuscript. All mathematical results and proofs were developed and verified independently by the authors.






\begin{thebibliography}{1}

    \bibitem{BM08} J. A. Bondy and U. S. R. Murty. Graph Theory. Springer, 2008.

\bibitem{CLM99} M. H. Carvalho, C. L. Lucchesi and U. S. R. Murty. Ear decompositions of matching covered graphs. Combinatorica, 19: 151-174, 1999.



    \bibitem{CLM02II} M. H. Carvalho, C. L. Lucchesi and U. S. R. Murty.
    On a conjecture of Lov\'{a}sz concerning bricks. \textrm{II}. Bricks of finite characteristic. J. Combin. Theory Ser. B, 85: 137-180, 2002.


    
    
\bibitem{Carvalho2004}
M. H. Carvalho, C. L. Lucchesi and U. S. R. Murty. Graphs with independent perfect matchings. J. Graph Theory, 48: 19-50, 2005.



    \bibitem{CLM12} M. H. Carvalho, C. L. Lucchesi and U. S. R. Murty. A generalization of Little's theorem on Pfaffian orientations. J. Combin. Theory Ser. B, 102: 1241-1266, 2012.
    

\bibitem{DL2026} X. Dai, F. Lu and Y. Zhang. Bricks in which every vertex is incident with a forcing edge. arXiv:2606.26594v2, 2026.

    \bibitem{ELP82} J. Edmonds, L. Lov\'asz and W. R. Pulleyblank. Brick decompositions and the matching rank of graphs. Combinatorica, 2(3): 247-274, 1982.




\bibitem{HLX25} X. He, F. Lu and J. Xue. Wheel-like bricks and minimal matching covered graphs. J. Graph Theory, 111(1): 5-16, 2026.

\bibitem{KCLL20} N. Kothari, M. H. Carvalho, C. L. Lucchesi and C. H. C. Little. On essentially 4-edge-connected cubic bricks. Electron. J. Combin., 27(1): \#P1.22, 2020.

    
    \bibitem{L1983} L. Lov\'{a}sz. Ear-decompositions of matching-covered graphs. Combinatorica, 3: 105-117, 1983.
    
     \bibitem{Lovasz87} L. Lov\'{a}sz. Matching structure and the matching lattice. J. Combin. Theory Ser. B, 43: 187-222, 1987.
 
\bibitem{LP86} L. Lov\'{a}sz and M. D. Plummer. Matching Theory. Number 29 in Annals of Discrete Mathematics. Elsevier Science, 1986. 


\bibitem{LuX24} F. Lu and J. Xue. Planar wheel-like bricks. arXiv:2410.20692, 2024.
    
    \bibitem{Lucchesi2024}
    C. L. Lucchesi and U. S. R. Murty. Perfect Matchings: A Theory of Matching Covered Graphs. Vol. 31. Springer, 2024.
    
    
    
 \bibitem{Tutte47} W. T. Tutte. The factorization of linear graphs. J. Lond. Math. Soc., 22:  107-111,  1947.
 
 \bibitem{LF23} X. Wu, F. Lu and L. Zhang. Removable edges in claw-free bricks. Graphs Combin., 40: 43, 2024. 




   

   
    



    
     \bibitem{Zhang24} Y. Zhang, F. Lu and H. Zhang. Cubic bricks that every $b$-invariant edge is forcing. arXiv:2411.17295, 2024.
     
     \bibitem{ZhangLuNearBipartite}
     Y. Zhang and F. Lu. Near-bipartite bricks in which every $b$-invariant edge is a forcing edge. arXiv:2607.00608, 2026.
     
  \bibitem{ZhangW25} Y. Zhang and X. Wang. Solid bricks that every $b$-invariant edge is solitary. arXiv:2507.21565, 2025.


     \bibitem{ZhangW25-1} Y. Zhang and X. Wang. Claw-free bricks that every $b$-invariant edge is solitary. arXiv:2509.23114, 2025.
     
     \bibitem{F2025} J. Zhou, X. Feng and W. Yan. Claw-free solid bricks. Graphs and Combinatorics, 41: 86, 2025.
     

   






%





   
    \end{thebibliography}
\end{document}